\documentclass[amssymb,twoside,12pt]{article}
\usepackage{latexsym,amsmath,graphicx,mathrsfs,amssymb}
\usepackage{amsfonts}
\usepackage{enumitem}
\newtheorem{theorem}{Theorem}[section]
\newtheorem{lemma}[theorem]{{\bf Lemma}}
\newtheorem{proposition}[theorem]{{\bf Proposition}}

\newtheorem{definition}{Definition}[section]
\numberwithin{equation}{section}
\newenvironment{proof}{\indent{\em Proof:}}{\quad \hfill
$\Box$\vspace*{2ex}}

\begin{document}
\setcounter{page}{1}
\begin{center}
\vspace{0.3cm} {\large{\bf On the uniqueness of mild solutions to the time-fractional Navier-Stokes 
equations in $L^{N} \left( \mathbb{R} ^{N}\right) ^{N} $}} \\
\vspace{0.4cm}
 J. Vanterler da C. Sousa $^1$ \\
vanterler@ime.unicamp.br  \\

\vspace{0.30cm}
Sadek Gala $^2$\\
sgala793@gmail.com\\

\vspace{0.30cm}
E. Capelas de Oliveira $^3$\\
capelas@ime.unicamp.br\\
\vspace{0.35cm}

\vspace{0.30cm}
$^{1,3}$ Department of Applied Mathematics, Imecc-Unicamp,\\ 13083-859, Campinas, SP, Brazil.\\

\vspace{0.30cm}
{$^{2}$ Department of Sciences Exactes\\
Ecole Normale Superieure of Mostaganem - University of Mostaganem\\
Mostaganem 27000, Algeria}
\end{center}

\def\baselinestretch{1.0}\small\normalsize
\begin{abstract}
In this paper, we present the result of maximum regularity of the mild solution of the fractional Cauchy problem. As our main result, we investigate the uniqueness of mild solutions for time-fractional Navier-Stokes equations in class $C\left([0,\infty);L^{N}\left( \mathbb{R}^{N}\right)^{N}\right)$ 
by means of the estimates $L^{p}-L^{q}$ of Giga-Shor inequality and the Gronwall inequality.
\end{abstract}
\allowdisplaybreaks
\section{Introduction}

Investigating Navier-Stokes equations has always been a
challenge for many researchers in the field of partial
differential equations, due to their importance and great relevance
\cite{inteiro1,inteiro2,inteiro3,navier,principal}. For instance, they 
are fundamental in modeling fluid behavior in
physical systems such as sea currents, blood flow and air masses, among
others.  Navier-Stokes equations form a system of non-linear differential
equations which still presents some open problems \cite{navier}. 
In order to the existence, uniqueness,
and regularity of solutions of Navier-Stokes equations,
we need some specific mathematical tools, which in turn
require great effort and dedication \cite{navier,mild2,mild3}.  A classic
example of this fact is that the existence of a mild global solution to the
three-dimensional equations for incompressible fluids remains
an open problem. 

Fractional calculus is also an important area of mathematics 
due to its 
well-founded theoretical basis, as well as its many applications
\cite{Jo2,Jo1,van,samko}. In recent times, 
researchers began to investigate the existence, uniqueness and regularity of mild solutions of time-fractional 
Navier-Stokes equations 
\cite{mild,mild1}. The project of unifying fractional calculus and Navier-Stokes equations is in fact 
something that is growing, and new 
works with interesting results are certainly to be expected. 

In 2015, Neto and Planas \cite{mild} wrote 
a work on mild solutions of time-fractional Navier-Stokes equations, in which they 
investigated the existence and uniqueness of mild solutions in $\mathbb{R}^{N}$. 
Peng et al. \cite{mild1}, in 2017, presented  
an excellent work on the properties of mild solutions of the time-fractional Navier-Stokes 
equations in Sobolev space via harmonic analysis. In the same year, Zhou and Peng \cite{weak}, established 
the existence and uniqueness of mild solutions (local and global) in $H^{\beta,q}$, for the Navier-Stokes 
equations with the Caputo fractional derivative of order $\alpha\in(0,1)$. In the same work, the authors 
investigated the existence and regularity of classical solutions. For a discussion of the results on 
solutions of Navier-Stokes equations using fractional derivatives, we suggest \cite{momani,wang12,Yukang,cholewa,ferreira,tang}.

In this paper, we consider the 
$N$-dimensional time-fractional Navier-Stokes equations in $\mathbb{R}^{N}$ $\left( N\geq 3\right) $, 
given by
\begin{equation}\label{eq1}
\left\{ 
\begin{array}{c}
\begin{array}{cll}
^{C}\mathbb{D}_{t}^{\alpha }\mathbf{u} & = & \Delta \mathbf{u}-\left( \mathbf{u}\cdot\nabla \right) 
\mathbf{u}-\nabla \mathbf{p} \\ 
\nabla \cdot \mathbf{u} & = & 0
\end{array}
\\ 
\left( x,t\right) \in \mathbb{R}^{N}\times \left( 0,T\right)\end{array}
\right.
\end{equation}
where $^{C}\mathbb{D}_{t}^{\alpha }\mathbf{u}\left( \cdot \right) $ is a Caputo fractional derivative 
of order $\alpha \in \left( 0,1\right)$, $\mathbf{u}=\mathbf{u}\left(x,t\right) :\mathbb{R}^{N}\times 
\mathbb{R} ^{+}\rightarrow \mathbb{R}^{N}$, $p\left( x,t\right) :\mathbb{R}^{N}\times \mathbb{R}^{+}
\rightarrow \mathbb{R}$ is the pressure (unknown), whose role is to maintain the divergence equal to 0, 
$ \nabla $ is the differential operator $\left( \partial _{x_{1}},...,\partial _{x_{N}}\right) $, 
$\nabla \cdot \mathbf{u}$ is the divergence of $\mathbf{u}$, $\Delta $ is Laplace operator, while 
$\left( \mathbf{u}\cdot\nabla \right) $ is the derivation operator 
$\mathbf{u}_{1}\partial_{x_{1}}+\mathbf{u}_{2}\partial _{x_{2}}+\cdot \cdot \cdot +
\mathbf{u}_{N}\partial_{x_{N}}$. We also have: $\left( \mathbf{u}\cdot\nabla \right) 
\mathbf{u}=\underset{j}{\sum }\partial _{j}\left( \mathbf{u}^{j}\mathbf{u}\right) $; 
$p=\left( -\Delta \right) ^{-1}\underset{j,k}{\sum }\partial _{j}\partial _{k}\left( 
\mathbf{u}^{j}\mathbf{u}^{k}\right)$; $\mathbb{P}=I_{d}-\nabla \Delta ^{-1}\nabla =I_{d}+R\otimes R$ 
where $R=\dfrac{1}{\sqrt{-\Delta }}\nabla $ is the Riesz transform and $R=\left( R_{1},...,R_{N}\right) ,$ 
$\widehat{R_{j}}f=i\dfrac{\xi _{j}}{\left\vert \xi \right\vert }\widehat{f}$ and $\mathbb{P}:L^{r}
\rightarrow L^{r}$ is the projector of Helmholtz-Leray.

Applying the projector $\mathbb{P}$ on both sides of the Eq.(\ref{eq1}) and
using the condition of divergence, we have $\mathbb{P}\mathbf{u}=\mathbf{u}$,
$\mathbb{P}$ $^{C}\mathbb{D}_{t}^{\alpha } \mathbf{u}=$
$^{C}\mathbb{D}_{t}^{\alpha }\mathbf{u}$, $\mathbb{P}\nabla \mathbf{p}=0$.
Substituting the term $\left( \mathbf{u} \cdot\nabla \right)
\mathbf{u}$ by $\nabla \cdot\left(\mathbf{u}\otimes \mathbf{u}\right) =\left(
\nabla  \cdot \mathbf{u}\right) \mathbf{u}+\left( \mathbf{u} \cdot\nabla
\right) \mathbf{u}$ (considered as a distribution), we then
have that the Cauchy problem for the incompressible time-fractional
Navier-Stokes equations in $\mathbb{R}^{N}$, can be rewritten as 
\begin{equation}\label{eq2}
\left\{ 
\begin{array}{rll}
^{C}\mathbb{D}_{t}^{\alpha }\mathbf{u}-\Delta \mathbf{u}+\mathbb{P}\nabla \cdot\left( \mathbf{u}\otimes \mathbf{u}\right) \mathbf{u} & = & 0,\,\,
\text{for }t\in \lbrack 0,T),\text{ }x\in \mathbb{R}^{N} \\ 
\nabla \cdot \mathbf{u} & = & 0,\text{ for }t>0,\text{ }x\in \mathbb{R}^{N} \\ 
\mathbf{u}\left( 0\right)  & = & \mathbf{u}_{0}.
\end{array}
\right. 
\end{equation}
Throughout the paper, we assume that the speed $\mathbf{u}_{0}$ satisfies $\nabla \cdot \mathbf{u_{0}} = 0\cdot T$
with $0< T \leq \infty$.

The Eq.(\ref{eq2}), in abstract form, is given by
\begin{equation}
\left\{ 
\begin{array}{cll}
^{C}\mathbb{D}_{t}^{\alpha }\mathbf{u} & = & A_{r}\mathbf{u}+F\left( \mathbf{u}\right) \\ 
\mathbf{u}\left( 0\right) & = & \mathbf{u}_{0}
\end{array}
\right.
\end{equation}
where $A_{r}\mathbf{u}=\Delta \mathbf{u}$ with $A_{r}:D\left( A_{r}\right) \subset 
L_{\sigma }^{r}\rightarrow L_{\sigma
}^{r}$ is the Stokes operator and $F\left( \mathbf{u}\right) =-\mathbb{P}\nabla \cdot\left( \mathbf{u}\otimes 
\mathbf{u}\right)$.

In what follows we investigate the
uniqueness of mild solutions for $N$-dimensional time-fractional Navier-Stokes
equations given by Eq.(\ref{eq2}) in order to provide new results for
this area and strengthen the link between fractional calculus
and partial differential equations, especially Navier-Stokes equations. In
addition, we demonstrate a result on maximum regularity for
the mild solution of the fractional Cauchy problem according to the Lemma
\ref{lema1}.

The paper is organized as follows. In section 2, we present the definition of
fractional Laplacian and the Gagliadro-Niremberg-Sobolev and Gronwall
inequalities; in addition, we present the definitions of Riemann-Liouville
fractional integral and Caputo fractional derivative.  We then show the mild
solution for the time-fractional Navier-Stokes equations given by the integral
equation; the solution is written in terms of the Mittag-Leffler functions of
one and two parameters. We investigate the maximum regularity of the mild
solution of the fractional Cauchy problem, that is, Lemma \ref{lema1}.  To
conclude the section, we present the proof of the Lemma \ref{lema2}, which is
fundamental to the proof of the main result of this paper. In section 3, we
investigate the uniqueness of the mild solutions of the time-fractional
Navier-Stokes equations written with the Caputo fractional derivative, using
the techniques presented in section 2.  Concluding remarks close the paper.

\section{Preliminary results}

Consider the Schwartz class, the class of $C^{\infty}$ functions on 
$\mathbb{R}^{N}$ whose derivatives decay faster than any polynomial. 


\begin{equation*} S:=\left\{ \mathbf{u}\in
	C^{\infty }\left( \mathbb{R}^{N}\right) :\underset{x\in
	\mathbb{R}^{N}}{\sup }\left\vert x^{\xi }\partial ^{\delta
	}\mathbf{u}\left( x\right) \right\vert <\infty ,\forall \xi ,\delta \in
\mathbb{N}^{N}\right\} . 
\end{equation*} 

\begin{definition}

Let $s \in \left( 0,1\right)$. The fractional Laplacian of order $s$ of the function $\mathbf{u}\in S,$ in which we denote by $\left( -\Delta \right) ^{s}\mathbf{u},$ is defined by {\rm\cite{mild}}
\begin{equation}\label{eq3}
\left( -\Delta \right) ^{s}\mathbf{u}\left( x\right) :=C\left( N,s\right) \text{ {\normalfont P.V}}\int_{\mathbb{R}^{N}}\frac{\mathbf{u}\left( x\right) -\mathbf{u}\left( y\right) }{\left\vert x-y\right\vert ^{N+2s}}dy
\end{equation}
where $C\left( N,s\right) :=\dfrac{2^{2s}s\Gamma \left( s+\frac{N}{2}\right) }{\pi ^{N/2}\Gamma \left( 1-s\right) }$ is a normalization constant.
\end{definition}

For a fixed $T>0$, we use the notation \cite{navier}
\begin{equation}\label{eq4}
\left\Vert h\right\Vert _{p,q,T}=\left( \int_{0}^{T}\left\Vert h\right\Vert_{L^{p}\left( \mathbb{R}^{N}\right) ^{N}}^{q}dt\right) ^{1/q},1\leq p\leq \infty ,\text{ }1\leq q\leq \infty 
\end{equation}
which denotes the standard space $L^{q}\left( \left( 0,T\right) ;L^{p}\left( \mathbb{R}^{N}\right) ^{N}\right)$ with the obvious modification if $q=\infty$.

We shall use the following inequality \cite{principal}:
\begin{equation}\label{va}
\left( a+b\right) ^{\beta }\leq 2^{\beta -1}\left( a^{\beta }+b^{\beta}\right) 
\end{equation}
for $a,b\geq 0$ and $\beta \geq 1$.

\begin{theorem}\label{gaglia}{\rm\cite{gagli}}{\rm (Gagliardo-Nirenberg-Sobolev
	inequality)} Assume that $1\leq p\leq N$. Then there exists a constant
	$C$ dependending only on $p$ and $N$ such that
\begin{equation*}
\left\Vert \mathbf{u}\right\Vert _{L^{\frac{pN}{N-p}\left( \mathbb{R}^{N}\right) }}\leq C\left\Vert \nabla \mathbf{u}\right\Vert _{L^{p}\left( \mathbb{R}^{N}\right) }
\end{equation*}
for all $\mathbf{u}\in C_{0}^{1}\left( \mathbb{R}^{N}\right)$.
\end{theorem}

\begin{theorem}\label{teo2}{\normalfont \cite{Jo1} (Gronwall inequality)} Let
	$u$ and $v$ be two integrable functions and $g$ a continuous function, with domain
	$[0,T]$. Let $\psi \in C^{1}\left[ 0,T\right] $ be an increasing
	function such that $\psi ^{\prime }\left( t\right) \neq 0,$ $ t\in
	\left[ 0,T\right]$. Assume that functions $u$ and $v$ are  non-negative and $g$
	is non-negative and non-decreasing. If
\begin{equation*}
\mathbf{u}\left( t\right) \leq \mathbf{v}\left( t\right) +g\left( t\right) \int_{0}^{T}\psi'\left( \tau \right) \left( \psi \left( t\right) -\psi \left( \tau \right)
\right) ^{\alpha-1}\mathbf{v}\left( \tau \right) d\tau
\end{equation*}
$t\in \lbrack 0,T)$, and as $v$ is a non-decreasing function over $\left[ 0,T\right]$, then
\begin{equation}\label{eq6}
\mathbf{u}\left( t\right) \leq \mathbf{v}\left( t\right) \mathbb{E}_{\alpha }\left( g\left( t\right)
\Gamma \left( \alpha \right) \left[ \psi \left( T\right) -\psi \left(
0\right) \right] ^{\alpha }\right) ,\forall t\in \left[ 0,T\right]
\end{equation}
where $\mathbb{E}_{\alpha }\left( \cdot \right) $ is a Mittag-Leffler function with one parameter, given by 
$\mathbb{E}_{\alpha }\left( t\right) =\overset{\infty }{\underset{k=0}{\sum }}\dfrac{t^{k}}{\Gamma \left( \alpha k+1\right) }$, with $0<\alpha<1$.

\end{theorem}

Let $h:\mathbb{R}^{N}\times \lbrack 0,T)\rightarrow \mathbb{R}^{N}$. The
Riemann-Liouville fractional integral of order $\alpha \in (0,1]$ of 
function $h$ is defined as \cite{Jo2,van,Jo3}

\begin{equation*}
I_{t}^{\alpha }h\left( x,t\right) =\frac{1}{\Gamma \left( \alpha \right) } \int_{0}^{t}\left( t-\tau \right) ^{\alpha -1}h\left( x,\tau \right) d\tau ,\text{ }t>0 .
\end{equation*}
Besides, the Caputo fractional derivative of order $\alpha $ of function $q$, is given by \cite{Jo2,van,Jo3}
\begin{equation*}
^{C}\mathbb{D}_{t}^{\alpha }h\left( x,t\right) :=\partial _{t}^{\alpha }h\left(
x,t\right) =\frac{1}{\Gamma \left( \alpha \right) }\int_{0}^{t}\left( t-\tau
\right) ^{\alpha -1}\frac{\partial }{\partial \tau}h\left( x,\tau \right) d\tau
,\text{ }t>0.
\end{equation*}

Let $M_{\alpha }$ be the Mainardi function, given by \cite{mild}
\begin{equation*}
M_{\alpha }\left( \theta \right) =\underset{k=0}{\overset{\infty }{\sum }}\frac{\theta ^{n}}{n!\Gamma \left( 1-\alpha \left( 1+n\right) \right) }.
\end{equation*}

This function is a particular case of Wright's function. The following
proposition presents a classical result about Mainardi function. 

\begin{proposition} For $\alpha \in \left( 0,1\right)$, $-1<r<\infty$ and $M_{\alpha }$ restricted to positive real line, $M_{\alpha }\left( t\right) \geq 0$ for all $t\geq 0$, we have
\begin{equation*} 
\int_{0}^{\infty
}t^{r}M_{\alpha }\left( t\right) dt=\frac{\Gamma \left( r+1\right) }{\Gamma
\left( \alpha r+1\right) }.
\end{equation*}
\end{proposition}

The mild solution for Eq.(\ref{eq2}), is given by the following integral
equation \cite{mild1}:
\begin{equation}\label{eq7}
\mathbf{u}\left( t\right) =\mathbb{E}_{\alpha }\left( t^{\alpha }\Delta \right) \mathbf{u}_{0}-\int_{0}^{t}\left( t-\tau \right) ^{\alpha -1}\mathbb{E}_{\alpha ,\alpha }\left( \left( t-\tau \right) ^{\alpha }\Delta \right) \mathbb{P}\nabla \cdot\left( \mathbf{u}\otimes \mathbf{u}\right) \left( \tau \right) d\tau 
\end{equation}
where
\begin{equation*}
\mathbb{E}_{\alpha }\left( t^{\alpha }\Delta \right) \mathbf{v}\left( x\right) =\left( \left( 4\pi t^{\alpha }\right) ^{-\frac{N}{2}}\int_{0}^{\infty }\theta ^{-\frac{N}{2}}M_{\alpha }\left( \theta \right) \exp \left( \frac{-\left\vert \cdot \right\vert ^{2}}{4\theta t^{2}}\right) d\theta \ast \mathbf{v}\right) \left(
x\right) 
\end{equation*}
and
\begin{equation*}
\mathbb{E}_{\alpha ,\alpha }\left( t^{\alpha }\Delta \right) \mathbf{v}\left( x\right) =\left( \left( 4\pi t^{\alpha }\right) ^{-\frac{N}{2}}\int_{0}^{\infty }\alpha \theta ^{1-\frac{N}{2}}M_{\alpha }\left( \theta \right) \exp \left( \frac{-\left\vert \cdot \right\vert ^{2}}{4\theta t^{2}}\right) d\theta \ast
\mathbf{v}\right) \left( x\right),
\end{equation*}
with $\mathbb{E}_{\alpha ,\beta }\left(t\right) =\overset{\infty
}{\underset{k=0}{\sum }}\dfrac{t^{k}}{\Gamma \left( \alpha k+\beta \right) }$, 
$0<\alpha<1$ and $0<\beta<1$.

The mild solution $u\in C\left( [0,T),L^{N}\left( \mathbb{R}^{N}\right) ^{N}\right) $ is associated with the initial condition $\mathbf{u}_{0}\in L^{N}\left(\mathbb{R}^{N}\right) ^{N}$ as $\nabla \cdot \mathbf{u}_{0}=0$.

Before investigating our main result, we need the results presented in Lemma \ref{lema1} and Lemma \ref{lema2}, below.

\begin{lemma}\label{lema1} Let $1<p,q<\infty $, $0<T<\infty$. If $h\in L^{q}\left( \left(0,T\right) ;L^{p}\left( \mathbb{R}^{N}\right) ^{N}\right)$, the function
\begin{equation}\label{eq8}
\mathbf{u}\left( t\right) =\mathbb{E}_{\alpha }\left( t^{\alpha }\Delta \right) \mathbf{u}_{0}+\int_{0}^{t}\left( t-\tau \right) ^{\alpha -1}\mathbb{E}_{\alpha ,\alpha }\left( \left( t-\tau \right) ^{\alpha }\Delta \right) \mathbb{P}h\left( \tau \right)d\tau 
\end{equation}
belongs to $L^{q}\left( \left( 0,T\right) ;L^{p}\left( \mathbb{R}^{N}\right) ^{N}\right) $ and solves the following Cauchy problem:
\begin{equation}\label{eq91}
\left\{ 
\begin{array}{rll}
^{C}\mathbb{D}_{t}^{\alpha }\mathbf{u}-\Delta \mathbf{u} & = & \mathbb{P}h\text{ for almost everywhere }t\in \left(
0,T\right)  \\ 
\mathbf{u}\left( 0\right)  & = & 0
\end{array}
\right. 
\end{equation}

In addition, the solution $u$ satisfies the estimate
\begin{equation}\label{eq9}
\left\Vert \Delta \mathbf{u}\right\Vert _{p,q,T}\leq C\left\Vert h\right\Vert _{p,q,T}
\end{equation}
with $C=C\left( p,N,q\right) >0$ independent of $h$ and $T$.
\end{lemma}

In the proof of Lemma \ref{lema1}, we will use the following definitions:
\begin{enumerate}
    \item $\Omega _{1}=\mathbb{R}^{N}$;
    \item $\Omega _{2}=$ limited domain;
    \item $\Omega _{3}=$ half space;
    \item $\Omega _{4}=$ external domain of $\mathbb{R}^{N}$.
\end{enumerate}

The proof shall be adapted from the proof of Theorem 2.7 \cite{giga}. The
result ensures that, if $\Omega \subseteq \mathbb{R}^{N}$ satisfies one of the
definitions $\Omega _{1}$-$\Omega _{4}$,  then the solution
$u$ of the Navier-Stokes equation is unique.

\begin{proof}
Indeed, as we have seen earlier, we have been able to rewrite the $N$-dimensional time-fractional Navier-Stokes equation Eq.(\ref{eq1}) in the form of Eq.(\ref{eq2}). On the other hand, Eq.(\ref{eq91}) can be written as follows: 
\begin{equation}\label{eq112}
\left\{ 
\begin{array}{rll}
^{C}\mathbb{D}_{t}^{\alpha }\mathbf{u}-\Delta \mathbf{u}+\nabla \mathbf{p} & = & h \\ 
\nabla \cdot\mathbf{u} & = & 0 \\ 
\mathbf{u}\left( x,0\right)  & = & \mathbf{u}_{0}.
\end{array}%
\right. 
\end{equation}

Next, we will use the embedding property for the second-order derivative $\Delta \mathbf{u}=\nabla ^{2}\mathbf{u}=\left( \partial _{j}\partial _{j}\mathbf{u}\right) ,$ $j=1,...,m,$
\begin{equation}\label{eq113}
\left\Vert \Delta \mathbf{u}\right\Vert _{p,q,T}=\left\Vert \nabla ^{2}\mathbf{u}\right\Vert _{p,q,T}\leq C\left\Vert A_{q}\right\Vert _{p,q,T},\text{ }\mathbf{v}\in D\left( A_{q}\right) 
\end{equation}
where $A_{q}=-\Delta \mathbf{u}.$

This applies to $C=C\left( p,q,N\right) >0$ for $1<p,q<\infty $ if $\Omega
	\subseteq \mathbb{R}^{N}$ satisfies one of the definitions $\Omega
	_{1}$-$\Omega _{3}$ and for $1<p,q<\frac{N}{2}$ in $\Omega _{4}$.

In fact, the result for the case $\Omega _{1}$ follows from Lemma 3.1
	\cite{23}; for $\Omega _{2}$, it follows from Lemma \ref{lema1}
	\cite{20,11}. The uniqueness in case $\Omega _{3}$ follows from Theorem
	3.6 \cite{3}. For the case $\Omega _{4}$, see \cite{42}.

In this sense, applying Eq.(\ref{eq113}) in Eq.(\ref{eq112}) and using
	$\nabla \mathbf{p}=g-$ $^{C}\mathbb{D}_{t}^{\alpha }+\Delta
	\mathbf{u}$, we obtain the following result:

\begin{eqnarray}
\left\Vert \Delta \mathbf{u}\right\Vert _{p,q,T} &\leq &\left\Vert ^{C}\mathbb{D}_{t}^{\alpha }\mathbf{u}\right\Vert _{p,q,T}+\left\Vert Ph\right\Vert _{p,q,T}  \notag \\ 
&\leq &C\left\Vert h\right\Vert _{p,q,T}.
\end{eqnarray}

This completes the demonstration.
\end{proof}
 
In the proof of Lemma 1 the authors Giga and Sohr assumed that $\Omega$
has an external domain, that is, a domain whose complement in $\mathbb{R}^{N}$
is a non-empty compact set. But since $\Omega = \mathbb{R}^{N}$ is all space,
Lemma \ref{lema1} has been proved following the same steps as Theorem 2.7
\cite{giga}.

\begin{lemma}\label{lema2}
	Let $g\in L^{q}\left( \left( 0,T\right) ;L^{p}\left(
	\mathbb{R}^{N}\right) ^{N^{2}}\right) $ where $1<p,q<\infty ,$
	$0<T<\infty $. Then, there exists a unique solution $\mathbf{v}=\left(
	-\Delta \right) ^{-1/2}\mathbf{u}$ belonging to $L^{q}\left( \left(
	0,T\right) ;L^{p}\left( \mathbb{R}^{N}\right) ^{N}\right) $ which
	solves the Cauchy problem
\begin{equation}\label{eq100}
\left\{ 
\begin{array}{rll}
^{C}\mathbb{D}_{t}^{\alpha }\mathbf{v}-\Delta \mathbf{v} & = & \mathbb{P}\left( -\Delta \right) ^{-1/2}\nabla \cdot h
\text{, almost everywhere }t\in \left( 0,T\right)  \\ 
\mathbf{v}\left( 0\right)  & = & 0,
\end{array}
\right.
\end{equation}
and satisfies the following estimates:
\begin{equation*}
\left\Vert \nabla \mathbf{u}\right\Vert _{p,q,T}\leq C\left\Vert h\right\Vert _{p,q,T}%
\end{equation*}
and
\begin{equation}\label{eq12}
\left\Vert \mathbf{u}\right\Vert _{\frac{pN}{N-p},q,T}\leq C\left\Vert h\right\Vert
_{p,q,T},\text{ }1<p<N , 
\end{equation}
with $C=C\left( p,N,q\right) >0$ independent of $h$ and $T$.
\end{lemma}

\begin{proof}
Applying $\left( -\Delta \right) ^{-1/2}$ in Eq.(\ref{eq91}), we have
\begin{equation}
^{C}\mathbb{D}_{t}^{\alpha }\mathbf{v}-\Delta \mathbf{v}=P\left( -\Delta \right) ^{-1/2}\nabla \cdot h.
\end{equation}

Thus, by the maximum regularity theorem in the fractional sense (Lemma
	\ref{lema1}), we know that there is a unique solution
	$\mathbf{v}\in L^{q}\left(\left( 0,T\right) ;L^{p}\left(
	\mathbb{R}^{N}\right) ^{N}\right) $ of Eq.(\ref{eq100}) for all $T>0$.

Moreover, from the Calderon-Zygmund theorem on singular integrals \cite{pod,pod1} and inequality (\ref{eq91}), we get
\begin{eqnarray}\label{eq118}
\left\Vert \nabla \mathbf{u}\right\Vert _{p,q,T} &=&\left\Vert \nabla \left( -\Delta
\right) ^{1/2}\mathbf{v}\right\Vert _{p,q,T}  \notag \\
&=&\left\Vert \Delta \mathbf{v}\right\Vert _{p,q,T}  \notag \\
&\leq &C\left\Vert h\right\Vert _{p,q,T}
\end{eqnarray}
using inequality (\ref{eq91}).

Using the inequality of Gagliardo-Nirenberg-Sobolev (Theorem \ref{gaglia}) and
	inequality (\ref{eq118}), we have that, for every $t\in \left[ 0,T\right]$, 
\begin{eqnarray}\label{palm}
\int_{0}^{T}\left\Vert \mathbf{u}\left( \tau \right) \right\Vert _{L^{\frac{pN}{N-p}}\left( \mathbb{R}
^{N}\right) }d\tau  &\leq &C\int_{0}^{T}\left\Vert \nabla \mathbf{u}\left( \tau
\right) \right\Vert _{L^{p}\left( \mathbb{R}^{N}\right) }^{q}d\tau \leq   \notag \\
&\leq &\widetilde{C}\int_{0}^{T}\left\Vert g\left( \tau \right) \right\Vert
_{L^{p}\left( \mathbb{R}^{N}\right) }^{q}d\tau .
\end{eqnarray}

Thus, raising both sides of this inequality to $1/q$, 
we conclude that

\begin{equation}
\left\Vert \mathbf{u}\right\Vert _{\frac{pN}{N-p},q,T} =\left(
	\int_{0}^{T}\left\Vert \mathbf{u}\left( \tau \right) \right\Vert
	_{L^{\frac{pN}{N-p}}\left( \mathbb{R}^{N}\right) }d\tau \right)
	^{1/q}\leq C\left( \int_{0}^{T}\left\Vert g\left( \tau \right)
	\right\Vert _{L^{p}\left( \mathbb{R}^{N}\right) }^{q}d\tau \right)
	^{1/q}=C\left\Vert g\right\Vert _{p,q,T}.
\end{equation}
\end{proof}


\section{Uniqueness of mild solution} 
In this section, we demonstrate the main
result of this paper, namely, the uniqueness of mild solution for
time-fractional Navier-Stokes equations Eq.(\ref{eq2}), by means of the estimates
in Lemma \ref{lema1} and Lemma \ref{lema2} and the
Gronwall inequality (Theorem \ref{teo2}).

\begin{theorem} Let $0<T\leq \infty $ and let $\mathbf{u},\mathbf{v}\in C\left(
[0,T);L^{N}\left( \mathbb{R}^{N}\right) ^{N}\right) $ be two solutions of the
time-fractional Navier-Stokes equation on $\left( 0,T\right) \times
\mathbb{R}^{N}$ with the same initial condition $\mathbf{u}_{0}$. Then,
	$\mathbf{u}=\mathbf{v} \in C [0,T)$. 
\end{theorem}

\begin{proof} For $\mathbf{u},\mathbf{v}\in C\left(
	[0,T);L^{N}\left(\mathbb{R}^{N}\right) ^{N}\right) $ and an $\varepsilon
	>0$, there are two decomposition 
	$\mathbf{u}=\mathbf{u}_{1}+\mathbf{u}_{2}$ and
	$\mathbf{v}=\mathbf{v}_{1}+\mathbf{v}_{2}$ such that, for every $T>0$,

\begin{equation}\label{eq13}
\left\Vert \mathbf{u}_{1}\right\Vert _{C\left( [0,T);L^{N}\left( \mathbb{R}^{N}\right) ^{N}\right) }\leq \varepsilon \text{ \ ; }\underset{\left(x,t\right) \in \mathbb{R}^{N}\times \left( 0,T\right) }{\sup }\left\vert \mathbf{u}_{2}\left( x,t\right)\right\vert <K\left( \varepsilon \right) 
\end{equation}
and
\begin{equation}\label{eq14}
\left\Vert \mathbf{v}_{1}\right\Vert _{C\left( [0,T);L^{N}\left( \mathbb{R}^{N}\right) ^{N}\right) }\leq \varepsilon \text{ \ ; }\underset{\left(x,t\right) \in \mathbb{R}^{N}\times \left( 0,T\right) }{\sup }\left\vert \mathbf{v}_{2}\left( x,t\right)\right\vert <K\left( \varepsilon \right).
\end{equation}

We can consider
\begin{equation*}
\mathbf{u}_{2}\left( x,t\right) =\left\{ 
\begin{array}{ccc}
\mathbf{u}\left( x,t\right) , & \text{for} & \left\vert \mathbf{u}\left( x,t\right)
\right\vert <K \\ 
0, & \text{for} & \left\vert \mathbf{u}\left( x,t\right) \right\vert \geq K%
\end{array}%
\right.
\end{equation*}
and
\begin{equation*}
\mathbf{v}_{2}\left( x,t\right) =\left\{ 
\begin{array}{ccc}
\mathbf{v}\left( x,t\right) , & \text{for} & \left\vert \mathbf{v}\left( x,t\right)
\right\vert <K \\ 
0, & \text{for} & \left\vert \mathbf{v}\left( x,t\right) \right\vert \geq K%
\end{array}%
\right. 
\end{equation*}
for a large enough $K$.

Now, assume that $\mathbf{u}$ and $\mathbf{v}$ are solutions in
$C\left([0,T);L^{N}\left( \mathbb{R}^{N}\right) ^{N}\right) $ with the
same initial conditions, for instance $\mathbf{u}\left( 0\right)
=\mathbf{v}\left( 0\right) =\mu$. Then, the difference $\xi =
\mathbf{u}-\mathbf{v}$ is a solution of the integral equation
\begin{equation*}
\xi\left( t\right) =-\int_{0}^{T}\left( t-\tau \right) ^{\alpha -1}\mathbb{E}_{\alpha ,\alpha }\left( \left( t-\tau \right) ^{\alpha }\Delta \right) \mathbb{P}\nabla \cdot \left( \xi\otimes \mathbf{u}+\mathbf{v}\otimes \xi\right) \left( \tau \right) d\tau .
\end{equation*}

Now, consider the functions
\begin{equation*}\label{eq16}
\xi_{1}\left( t\right) =-\int_{0}^{T}\left( t-\tau \right) ^{\alpha -1}\mathbb{E}_{\alpha ,\alpha }\left( \left( t-\tau \right) ^{\alpha }\Delta \right) \mathbb{P}\nabla \cdot\left( \xi\otimes \mathbf{u}_{1}+\mathbf{v}_{1}\otimes \xi\right) \left( \tau \right) ds 
\end{equation*}
and
\begin{equation*}\label{eq17}
\xi_{2}\left( t\right) =-\int_{0}^{T}\left( t-\tau \right) ^{\alpha -1}\mathbb{E}_{\alpha ,\alpha }\left( \left( t-\tau \right) ^{\alpha }\Delta \right) \mathbb{P}\nabla \cdot\left( \xi\otimes \mathbf{u}_{2}+\mathbf{v}_{2}\otimes \xi\right) \left( \tau \right) ds.
\end{equation*}

The convolution operator $\mathbb{E}_{\alpha ,\alpha
}\left( \left( t-\tau \right) ^{\alpha }\Delta \right)
\mathbb{P}\nabla$ has an integrable core whose standard is $O\left(
\left( t-\tau \right) ^{-\alpha /2}\right) $ in $L_{1}$. From this property and using the
estimates Eq.(\ref{eq13}) and Eq.(\ref{eq14}) and H\"older's inequality
repeatedly in time, we have that 

\begin{eqnarray}\label{eq18}
\left\Vert \xi_{2}\left( t\right) \right\Vert _{L^{N}} &\leq &C\int_{0}^{T}\left( t-\tau \right) ^{\frac{\alpha }{2}-1}\left\Vert
\xi\left( \tau \right) \right\Vert _{L^{N}}\left( \left\Vert \mathbf{u}_{2}\left( \tau
\right) \right\Vert _{L^{\infty }}+\left\Vert \mathbf{v}_{2}\left( \tau \right)
\right\Vert _{L^{\infty }}\right) d\tau   \notag \\
&\leq &2CK\left( \varepsilon \right) \left( \int_{0}^{T}\left( t-\tau
\right) ^{\frac{2}{3}\left( \alpha -2\right) }d\tau \right) ^{3/4}\left(
\int_{0}^{T}\left\Vert \xi\left( \tau \right) \right\Vert _{L^{N}}^{4}d\tau
\right) ^{1/4}  \notag \\
&\leq &2CK\left( \varepsilon \right) t^{\frac{2\alpha -1}{4}}\left(
\int_{0}^{T}\left\Vert \xi\left( \tau \right) \right\Vert _{L^{N}}^{4}d\tau
\right) ^{1/4} , 
\end{eqnarray}
where $C$ denotes a constant independent of $\xi,t$.

Now, raising both sides of inequality (\ref{eq18}) to the fourth power and taking the integral with respect to $\tau \in \left( 0,T\right)$, we have
\begin{equation}\label{eq19}
\int_{0}^{T}\left\Vert \xi_{2}\left( \tau \right) \right\Vert _{L^{N}}^{4}d\tau \leq 2C^{4}\left( K\left( \varepsilon \right) \right) ^{4}T^{2\alpha -1}\int_{0}^{T}\left( \int_{0}^{\tau}\left\Vert \xi\left( s
\right) \right\Vert _{L^{N}}^{4}ds \right) d\tau .
\end{equation}

On the other hand, by estimate Eq.(\ref{eq12}) of Lemma \ref{lema2}, estimates
Eq.(\ref{eq13}) and Eq.(\ref{eq14}), H\"older's inequality and inequality
Eq.(\ref{va}), we obtain

\begin{eqnarray}\label{eq20}
\int_{0}^{T}\left\Vert \xi_{1}\left( \tau \right) \right\Vert _{L^{N}}^{4}d\tau  &\leq &C\int_{0}^{T}\left\Vert \left( \xi\otimes \left( \mathbf{u}_{1}+\mathbf{v}_{1}\right) \right) \left( \tau \right) \right\Vert _{L^{\frac{N}{2}}}^{4}d\tau   \notag \\
&\leq &C\left( \left\Vert \mathbf{u}_{1}\left( \tau \right) \right\Vert _{C\left(
[0,T);L^{N}\left( \mathbb{R}^{N}\right) ^{N}\right) }+\left\Vert \mathbf{v}_{1}\left( \tau \right) \right\Vert
_{C\left( [0,T);L^{N}\left( \mathbb{R}^{N}\right) ^{N}\right) }\right) \int_{0}^{T}\left\Vert \xi\left( \tau \right) \right\Vert _{L^{N}}^{4}d\tau   \notag \\
&\leq &2\varepsilon C\int_{0}^{T}\left\Vert \xi\left( \tau \right) \right\Vert
_{L^{N}}^{4}d\tau.
\end{eqnarray}

Taking $\varepsilon $ small, we have from inequalities Eq.(\ref{eq19}) and
Eq.(\ref{eq20}) that 

\begin{equation*}
\int_{0}^{T}\left\Vert \xi\left( \tau \right) \right\Vert _{L^{N}}^{4}d\tau \leq \widetilde{C_{\alpha}}\left( K\left( \varepsilon \right) \right) ^{4}  T^{2\alpha-1} \int_{0}^{T}(t-\tau)^{\alpha-1}\left( \int_{0}^{\tau }\left\Vert \xi\left( s\right) \right\Vert _{L^{N}}^{4}ds\right) d\tau 
\end{equation*}
$0\leq \tau \leq T.$

Using the Gronwall inequality (Theorem \ref{teo2}), we finally have
\begin{equation*}
\int_{0}^{T}\left\Vert \xi\left( \tau \right) \right\Vert _{L^{N}}^{4}d\tau \leq 0.\mathbb{E}_{\alpha }\left( \widetilde{C_{\alpha}}\left( K\left( \varepsilon \right) \right) ^{4}T^{3\alpha -1}\Gamma(\alpha)\right) =0,
\end{equation*}
which implies that $\displaystyle\int_{0}^{T}\left\Vert \xi\left( \tau \right) \right\Vert _{L^{N}}^{4}d\tau =0\Longleftrightarrow \xi=0$. Therefore, $\mathbf{u}=\mathbf{v}$.
\end{proof}


\section{Concluding remarks}

We investigated the uniqueness of mild solution for time-fractional
Navier-Stokes equations in $L^{N} \left( \mathbb{R} ^{N}\right) ^{N} $ by means
of 
estimates (Lemma \ref{lema1} and Lemma \ref{lema2}) and the Gronwall inequality.
A direct consequence of the results obtained here is that when $\alpha = 1$, we
recover the result valid for the classical Navier-Stokes equation. It is worth
mentioning that it remains an open problem 
the investigation of the existence, uniqueness and regularity
of mild solutions for time-fractional Navier-Stokes equations introduced by
$\psi$-Caputo fractional derivative \cite{Jo2}. 
It seems that a possible way to approach this open problem would be to introduce a
new Laplace transform involving the derivative of a function taken in relation to
another function.

\section*{Acknowledgment}

JVCS acknowledges the financial support of a PNPD-CAPES (process number
88882.305834/2018-01) scholarship of the Postgraduate Program in Applied
Mathematics of IMECC-Unicamp. We are grateful to Dr. J. Em\'{\i}lio Maiorino 
for many and usefull discussions.



\begin{thebibliography}{99}

\bibitem{mild2} Giga, Y. and Sohr, H., Abstract ${L}^{p}$ estimates for the {C}auchy problem with applications to the {N}avier-{S}tokes equations in exterior domains, J. Funct. Anal., 102(1), 72--94, (991).

\bibitem{mild3} Monniaux, S., Uniqueness of mild solutions of the {N}avier-{S}tokes equation and maximal ${L}^{p}$-regularity, Comptes Rendus de l'Acad{\'e}mie des Sciences-Series I-Mathematics, 328(8), 663--668, (1999).

\bibitem{mild} de Carvalho-Neto, P. M. and Planas, G., Mild solutions to the time fractional {N}avier-{S}tokes equations in $\mathbb{R}^{N}$, J. Diff. Equ., 259(7), 2948--2980, (2015).

\bibitem{mild1} Peng, L. and Zhou, Y. and Ahmad, B. and Alsaedi, A., The {C}auchy problem for fractional {N}avier-{S}tokes equations in {S}obolev spaces, Chaos, Solitons \& Fractals, 102, 218--228, (2017).


\bibitem{navier} Sohr, H., The {N}avier-{S}tokes {E}quations: {A}n {E}lementary {F}unctional {A}nalytic {A}pproach, 2001, Birkhäuser Advanced Texts, Basel.

\bibitem{Jo1} Vanterler da C. Sousa, J. and Capelas de Oliveira, E., A {G}ronwall inequality and the {C}auchy-type problem by means of $\psi$-{H}ilfer operator, Diff. Equ. \& Appl., 11(1), 87--106 (2019) .

\bibitem{Jo2}Vanterler da C. Sousa, J. and Capelas de Oliveira, E., On the $\psi$-{H}ilfer fractional derivative, Commun. Nonlinear Sci. Numer. Simul., 60, 72--91, (2018).

\bibitem{van} Vanterler da C. Sousa, J. and Capelas de Oliveira, E., Leibniz type rule: $\psi$-Hilfer fractional operator, Commun. Nonlinear Sci. Numer. Simul., 77, 305--311, (2019).

\bibitem{principal} Gala, S., A note on the uniqueness of mild solutions to the {N}avier-{S}tokes equations, Archiv der Mathematik, 88(5), 448--454, (2017).

\bibitem{gagli}Dolbeault, J. and Esteban, M. J. and Laptev, A. and Loss, M., One-dimensional {G}agliardo-{N}irenberg-{S}obolev inequalities: remarks on duality and flows, J. London Math. Soc., 90(2), 525--550, (2014).

\bibitem{Jo3} Vanterler da C. Sousa, J. and Capelas de Oliveira, E., On the $\Psi$-fractional integral and applications, Comp. Appl. Math., 38(1), 4, (2018).

\bibitem{haide}Gou, H. and Li, B., Study on the mild solution of {S}obolev type Hilfer fractional evolution equations with boundary conditions, Chaos, Solitons \& Fractals, 112, 168--179, (2018).

\bibitem{zhou}Zhou, Y. and Peng, L., On the time-fractional {N}avier-{S}tokes equations, Comp. \& Math. Appl., 73(6), 874--891, (2017).

 \bibitem{weak} Zhou, Y. and Peng, L., Weak solutions of the time-fractional {N}avier-{S}tokes equations and optimal control, Comp. \& Math. Appl., 73(6), 1016--1027, (2017).

\bibitem{momani}Momani, S. and Odibat, Z., Analytical solution of a time-fractional {N}avier-{S}tokes equation by {A}domian decomposition method, Appl. Math. Comp., 177(2), 488--494, (2006).

\bibitem{wang12}Wang, Y and Liang, T., Mild solutions to the time fractional {N}avier-{S}tokes delay differential inclusions, Discrete \& Continuous Dynamical Systems-B, 439--467, (2018).

\bibitem{Yukang}Chen, Y. and Wei, C., Partial regularity of solutions to the fractional {N}avier-{S}tokes equations, Discrete \& Continuous Dynamical Systems-A, 36 (10), 5309--5322, (2016).

\bibitem{cholewa}Cholewa, J. W. and Dlotko, T., Fractional {N}avier-{S}tokes equations, Discrete \& Continuous Dynamical Systems-Series B, 23(8), 2967--2988, (2018).

\bibitem{ferreira}Ferreira, L. and Villamizar-Roa, E. J., Fractional {N}avier-{S}tokes equations and a {H}{\"o}lder-type inequality in a sum of singular spaces, Nonlinear Analysis: Theory, Methods \& Applications, 74(16), 5618--5630, (2011).

\bibitem{tang} Tang, L. and Yu, Y., Partial H{\"o}lder regularity of the steady fractional {N}avier-{S}tokes equations, Calc. Var. Partial Diff. Equ., 55(2), 31, (2016).

\bibitem{inteiro1}Khai, D. Q. and Tri, N. M., Well-posedness for the {N}avier-{S}tokes equations with data in homogeneous {S}obolev-{L}orentz spaces, Nonlinear Analysis: Theory, Methods \& Applications, 149, 130--145, (2017).

\bibitem{inteiro2}Danchin, R., Global existence in critical spaces for compressible {N}avier-{S}tokes equations, Inventiones Mathematicae, 141(3), 579--614, (2000).

\bibitem{inteiro3}Feireisl, E. and Novotn{\`y}, A. and Petzeltov{\'a}, H., On the existence of globally defined weak solutions to the {N}avier-{S}tokes equations, J. Math. Fluid Mechanics, 3(4), 358--392, (2001).

\bibitem{giga}Giga, Y. and Sohr, H., Abstract ${L}^{p}$ estimates for the {C}auchy problem with applications to the {N}avier-{S}tokes equations in exterior domains, J. Functional Anal., 102(1), 72--94, (1991).

\bibitem{23}Giga, Y. and Sohr, H., Note on the {C}auchy problem in {B}anach spaces with applications to the {N}avier-{S}tokes equation in exterior domains, unpublished preprint, (1988).

\bibitem{20}Giga, Y., Solutions for semilinear parabolic equations in ${L}^{p}$ and regularity of weak solutions of the {N}avier-{S}tokes system, J. Diff. Equ., 62(2), 186--212, (1986).

\bibitem{11}Cattabriga, L., Su un problema al contorno relativo al sistema di equazioni di {S}tokes, Rendiconti del Seminario Matematico della Universita di Padova, 31, 308--340, (1961).

\bibitem{42}Solonnikov, V. A., Estimates for solutions of non stationary {N}avier-{S}tokes equations, J. Math. Sci., 8(4), 467--529, (1977).

\bibitem{3}Borchers, W. and Miyakawa, T., ${L}^{2}$ decay for the {N}avier—{S}tokes flow in half spaces, Math. Ann., 282(1), 139--155, (1988).

\bibitem{samko}Kilbas, A. A. and Srivastava, H. M. and Trujillo, J. J., Theory and {A}pplications of {F}ractional {D}ifferential {E}quations, 2006, Elsevier, San Diego.



\bibitem{pod}Krylov, N. V., The {C}alder{\'o}n-{Z}ygmund theorem and parabolic equations in ${L}_{p}(\mathbb{R}, {C}^{2+\alpha}) $-spaces, Annali della Scuola Normale Superiore di Pisa-Classe di Scienze, 1(4), 799--820, (2002).

\bibitem{pod1}Calder{\'o}n, A. P. and Zygmund, A., On singular integrals, Amer. J. Math., 78(2), 289--309, (1956).

\end{thebibliography}
\end{document}